\documentclass[psamsfonts]{amsart}

\usepackage{amssymb,amsfonts}
\usepackage[all,arc]{xy}
\usepackage{enumerate}
\usepackage{mathrsfs}

\newtheorem{thm}{Theorem}[section]
\newtheorem{cor}[thm]{Corollary}
\newtheorem{prop}[thm]{Proposition}
\newtheorem{lem}[thm]{Lemma}

\theoremstyle{definition}

\theoremstyle{remark}

\makeatletter
\let\c@equation\c@thm
\makeatother
\numberwithin{equation}{section}

\title{A cohomological property of semi-abelian $p$-groups}

\author{ Mohammed T. Benmoussa and Yassine Guerboussa$^{*}$}

\date{}

\begin{document}

\begin{abstract}
We prove a cohomological property for a class of finite $p$-groups introduced earlier by M. Y. Xu, which we call semi-abelian $p$-groups.  This result implies that a semi-abelian $p$-group has non-inner automorphisms of order $p$, which settles a longstanding problem for this class.  We answer also, independetly, an old question of M. Y. Xu about the power structure of semi-abelian $p$-groups.

\end{abstract}
\thanks{{\scriptsize
\hskip -0.4 true cm 
\newline Keywords:  cohomology; finite $p$-groups; automorphisms.\\
$*$Corresponding author}}
\maketitle

\begin{center}
Dedicated to Professor Ming-Yao Xu for his early work on finite $p$-groups.
\end{center}
\section{Introduction}

Let $G$ be a finite $p$-group.
Following M. Y. Xu (see \cite{Xu2}), we say that $G$ is {\itshape strongly semi- $p$-abelian}, if the following property holds in $G$ :
$$(xy^{-1})^{p^n}=1 \Leftrightarrow x^{p^n}=y^{p^n} \mbox{ for any postive integer } n.$$
For brevity, we shall use the term {\itshape semi-abelian} for such a group. It is easy to see that such a group satisfies the properties : 

\begin{enumerate}[(i)]
 
\item $\Omega_{n}(G)=\Omega_{\{n\}}(G)$.
\item $|G:G^{\{p^n\}}|=|\Omega_{n}(G)|$, and so $|G:G^{p^n}| \leq |\Omega_{n}(G)|$.
\end{enumerate}

Hence, semi-abelian $p$-groups share some nice properties with the regular $p$-groups introduced by P. Hall (see  \cite{Berk1} for their theory).  It is not difficult to show that every regular $p$-group is semi-abelian; however the class of semi-abelian $p$-groups is much larger, and in fact every finite $p$-group can occur as a quotient of a semi-abelian $p$-group (see Section 3).

Let $G$ be a regular $p$-group, and $1\textless N \lhd G$ such that $G/N$ is not cyclic.  P. Schmid showed in \cite{Sch}, that the Tate cohomology groups $\hat{H}^n(G/N,\operatorname{Z}(N))$ are all non-trivial; where $\operatorname{Z}(N)$ is considered as a $G/N$-module with the action induced by conjugation in $G$.  
Our first purpose is to show that Schmid's result holds in a more general context. 

\begin{thm}\label{main}
Let $G$ be a semi-abelian $p$-group, and $1\textless  N \lhd G$ such that $G/N$ is neither cyclic nor a generalized quaternion group.  Then $\hat{H}^n(G/N,\operatorname{Z}(N)) \neq 0$, for all integers $n$.  
\end{thm}   

Let us note that P. Schmid has conjuctured that Theorem \ref{main} holds for an arbitrary finite p-group $G$ if one takes $N= \Phi (G)$. This conjecture has been refuted by A. Abdollahi in \cite{Abd}.  Although, it is interesting to find other classes of $p$-groups which satisfy the conclusion of Theorem \ref{main} (see \cite[Question 1.2]{Abd2}).\\

The above theorem is intimately related to studying non-inner automorphisms of finite $p$-groups.  An abelian normal subgroup $A$ of a group $G$ can be seen as a $G$-module via conjugation,  whence we can consider the group of crossed homomorphisms or derivations $\operatorname{Der}(G,A)$.  To each derivation $\delta \in \operatorname{Der}(G,A)$, we can associate an endomorphism $\phi_{\delta}$ of $G$, given by $\phi_{\delta}(x)=x \delta(x)$, $x\in G$. This map $\phi$ sends $\operatorname{Der}(G,A)$ into 
$$\operatorname{End}_A(G) =\{\theta \in \operatorname{End}(G)\, | \, x^{-1}\theta(x) \in A, \mbox{ for all }x\in G \}$$
 and in fact it defines a bijection between the two sets $\operatorname{Der}(G,A)$ and $\operatorname{End}_A(G)$.

If we consider only the set of derivations $\delta :G \rightarrow A$, that are trivial on $\operatorname{C}_G(A)$ (i.e. $\delta(x)=1$, for $x\in \operatorname{C}_G(A)$), which can be identified to $\operatorname{Der}(G/\operatorname{C}_G(A),A)$; then the map $\phi$ induces an isomorphism between $\operatorname{Der}(G/\operatorname{C}_G(A),A)$ and the group $\operatorname{\widetilde{C}} (A)$ of the automorphisms of $G$ acting trivially on $\operatorname{C}_G(A)$ and  $G/A$. 

It is straightforward to see that this isomorphism maps $\operatorname{Ider}(G/\operatorname{C}_G(A),A)$ into a group of inner automorphisms lying in $\operatorname{\widetilde{C}} (A)$, though an inner automorphism lying in $\operatorname{\widetilde{C}} (A)$ needs not necessarily be induced by an inner derivation; however this case can be avoided by assuming that $\operatorname{C}_G(\operatorname{C}_G(A))=A$, as if $\phi_{\delta}(x)=x^g$ for some $g \in G$ and all $x\in G$, then $g \in \operatorname{C}_G(\operatorname{C}_G(A))$, so $g$ lies in $A$ and $\delta$ is the inner derivation induced by $g^{-1}$.  We have established
\begin{prop}\label{Cores}
Let $G$ be a group, and $A$ be an abelian normal subgroup of $G$ such that $\operatorname{C}_G(\operatorname{C}_G(A))=A$; let $\operatorname{\widetilde{C}} (A)$ denote the group of the automorphisms of $G$ acting trivially on $\operatorname{C}_G(A)$ and  $G/A$.  Then there is an isomorphism from $\operatorname{Der}(G/\operatorname{C}_G(A),A)$ to $\operatorname{\widetilde{C}} (A)$, which maps $\operatorname{Ider}(G/\operatorname{C}_G(A),A)$ exactly to the inner automorphisms lying in $\operatorname{\widetilde{C}} (A)$.  
In particular if $\operatorname{\widetilde{C}} (A) \leq \operatorname{Inn} (G)$, then $\hat{H}^1(G/\operatorname{C}_G(A),A)=0$.   
\end{prop}              
This well known fact in the litterature, which can be found for instance in \cite{Gru}, permits to reduce the problem of existence of non inner automorphisms of some group to a cohomological problem. For instance, this allowed W. Gasch\"{u}tz to prove that any non simple finite $p$-group has non inner automorphisms of p-power order.

It is conjectured by Y. Berkovich that a more refined version of Gasch\"{u}tz's result holds, more precisely that a non simple finite $p$-group has non inner automorphisms of order $p$ (see \cite[Problem 4.13]{Kou}).  While it is not clear that a positive answer to it, has deep implications for our understanding of finite $p$-groups, this problem received a large interest, and its hardness may stimulates further developments of new techniques in  finite $p$-group theory.  The reader may find more information and the relevant references about this problem in \cite{Abd2}.

Our second result settles this problem in the class of semi-abelian $p$-groups.
\begin{thm}\label{main2}
Let $G$ be a semi-abelian finite $p$-group.  Then $G$ has a non inner automorphism of order $p$. 
\end{thm}

Our notation is standard in the litterature.  Let $S$ be a group.  For a positive integer $n$,  we denote by $S^{\{p^n\}}$ the set of  the $p^n$-th powers of all the  elements of $S$ and by $S^{p^n}$ the subgroup generated by $S^{\{p^n\}}$.  The subgroup generated by the elements of order dividing $p^n$ is denoted by $\Omega_{n}(S)$.  We denote by $\lambda_{n}^p(S)$ the terms of the lower $p$-series of $S$ which are defined inductively by :
$$\lambda_{1}^p(S)=S \mbox{, and } \lambda_{n+1}^p(S)=[\lambda_{n}^p(S),S]\lambda_{n}^p(S)^p.$$ 
If $A$ is an $S$-module, then $A_S$ denotes the subgroup of fixed elements in $A$ under the action of $S$.

The remainder of the paper is divided into two sections.  In Section 2, we prove Theorem \ref{main} and Theorem \ref{main2};  and in Section 3 we answer an old question of M. Y. Xu (see \cite[Problem 3]{Xu2}) about the power structure of semi-abelian $p$-groups.  This result follows quickly from a result of D. Bubboloni and  G. Corsi Tani (see \cite{Bub}), but it seems not that this link has been noted before. 
\section{Proofs}
Let $Q$ be a finite $p$-group, and $A$ be a $Q$-module of $p$-power order.  Recall that $A$ is said to be cohomologically trivial if $\hat{H}^k(S,A)=0$ for all $S \leq Q$ and all integers $k$.  

It is proved by W. Gasch\"{u}tz and K. Ushida (independently)  that $\hat{H}^1(Q,A)=0$ implies that $\hat{H}^k(S,A)=0$ for all $S \leq Q$ and all integers $k \geq 1$ (see \cite[ Lemma 2, \S  7.5 ]{Gru}).    This statement can be slightly improved as noted in \cite{Hoe}.    

\begin{prop}\label{Gas}
Let $Q$ be a finite $p$-group, and $A$ be a $Q$-module which is also a finite $p$-group.  If $\hat{H}^n(Q,A)=0$ for some integer $n$, then  $A$ is cohomologically trivial.  
\end{prop}
We shall use Proposition \ref{Gas} to reduce the proof of Theorem \ref{main} to the non-vanishing of the Tate cohomology groups in dimension $0$, which is easier to handle.

We need also the following result of Schmid (see \cite[Proposition 1]{Sch}).
\begin{prop}\label{Schmid}
Let $Q$ be a finite $p$-group, and $A \neq 1$ be a $Q$-module which is also a finite $p$-group.  If  $A$ is cohomologically trivial, then $\operatorname{C}_Q(A_K)=K$, for every $K \leq Q$.  
\end{prop}
We need also to prove the following

\begin{lem}\label{L}
Under the assumption of Theorem \ref{main}, set $A=\operatorname{Z}(N)$ and  let $S/N$ be a subgroup of exponent $p$ of $G/N$.  Then $A^p \leq A_{S/N}$, so that $\operatorname{C}_{S/N}(A^p)=S/N$.
\end{lem}

\begin{proof}
Let be $x \in S$ and $a \in A$.  We have $x^p \in N$, hence 
$$x^p = (x^p)^a=(x^a)^p=(x[x,a])^p.$$
As $G$ is semi-abelian, we have $[x,a]^p=1$.  It follows that $(a^{-1}a[a,x])^p=[a,x]^p=1$, and again since $G$ is semi-abelian, we have 
$$a^p=(a[a,x])^p=(a^x)^p=(a^p)^x.$$
This shows that $A^p$ is centralized by every element of $S/N$.           
\end{proof}

\begin{proof}[Proof of Theorem \ref{main}]

Assume for a contradiction 
that $\hat{H}^n(G/N,A)=0$ for some integer $n$, where $A$ denotes $\operatorname{Z}(N)$.  As $G/N$ is not cyclic and different from the generalized quternion groups $Q_{2^m}$, there is in $G/N$ a subgroup $S/N$ of exponent $p$ and order $\geq p^2$.  It follows from Proposition \ref{Gas},   that $\hat{H}^n(S/N,A) = 0$, so $A$ is a cohomologically trivial $S/N$-module.  Let $K/N \leq S/N$ be a subgroup of order $p$.  Proposition \ref{Gas} implies that $\hat{H}^0(K/N,A) = 0$.  We have $\hat{H}^0(K/N,A) = A_{K/N}/A^{\tau}=0$, where $A^{\tau}$ is the image of $A$ under the trace homomorphism $\tau : A \rightarrow A$ induced by $K/N$.  As $K/N$ is cyclic of order $p$, our trace map is given by 

$$a^{\tau}= aa^x\dots a^{x^{p-1}}\mbox{ for } a\in A \mbox{,  and any fixed }x\in K-N$$ 
from which it follows 
$$a^{\tau}=(ax^{-1})^{p}x^p.$$
Now as $G$ is semi-abelian, $a \in \ker\tau$ if, and only if $a^p=1$; that is $\operatorname{ker}\tau=\Omega_1(A)$.  This implies that $|A^{\tau}|=|A^p|$.  As $A_{K/N}=A^{\tau}$, and $A^p \leq A_{K/N}$ by Lemma \ref{L}, we have $A^p=A_{K/N}$.  By Proposition \ref{Schmid}, $\operatorname{C}_{S/N}(A^p)=\operatorname{C}_{S/N}(A_{K/N})=K/N$, however Lemma \ref{L} implies that $S/N=K/N$, a contradiction.   
      
\end{proof}
Before proving Theorem \ref{main2}, we need the following reduction from \cite{Deac}.
\begin{prop}\label{Deac}
Let $G$ be a finite $p$-group such that   $\operatorname{C}_G(\operatorname{Z}(\Phi(G)) \neq \Phi(G)$.  Then $G$ has a non inner automorphism of order $p$.
\end{prop}
Note that M. Ghoraishi improved Proposition \ref{Deac} in \cite{Gho}, where he reduced  the problem of Berkovich to the  $p$-groups $G$ satisfying $H \leq \operatorname{C}_G(H) = \Phi(G)$, where $H$ is the inverse image of $\Omega_1(\operatorname{Z}(G/\operatorname{Z}(G)))$ in $G$.  A family of examples which satisfy the condition $\operatorname{C}_G(\operatorname{Z}(\Phi(G)) = \Phi(G)$ and do not satisfy Ghoraishi's condition can be found in the same paper.   

\begin{proof}[Proof of Theorem \ref{main2}]
Assume for a contradiction that every automorphism of $G$ of order $p$ is inner.  Let be $A=\operatorname{Z}(\Phi(G))$.  By Proposition \ref{Deac}, we have $\operatorname{C}_G(A)=\Phi(G)$ and so $\operatorname{C}_G(\operatorname{C}_G(A))=A$.  If we prove that $\operatorname{Der}(G/\operatorname{C}_G(A),A)=\operatorname{Der}(G/\Phi(G),\operatorname{Z}(\Phi(G))$ has exponent $p$, then our first assumption together with Proposition \ref{Cores} imply that $\hat{H}^1(G/\Phi(G),\operatorname{Z}(\Phi(G)))=0$, which contradicts Theorem \ref{main}.  So we need only to prove, for any derivation $\delta \in \operatorname{Der}(G,\operatorname{Z}(\Phi(G))$ which  is trival on $\Phi(G)$, that    $\delta(x)^p=1, \mbox{ for all } x\in G$.  Indeed 
$$\delta(x^p)=\delta(x) \delta(x)^x \dots \delta(x)^{x^{p-1}}=(\delta(x)x^{-1})^px^p$$
As $\delta$ is trivial on $\Phi(G)$, we have $\delta(x^p)=(\delta(x)x^{-1})^px^p=1$, and since $G$ is semi-abelian it follows that $\delta(x)^p=1$.   
\end{proof}
\section{Remarks on a particular class of semi-abelian $p$-groups}
 M. Y. Xu  proved in \cite{Xu1}, that any finite $p$-group $G$, $p$ odd, which satisfies $\Omega_1(\gamma_{p-1}(G))\leq \operatorname{Z}(G)$ is semi-abelian; and he asked if such a group must be power closed, that is every element of $G^{p^n}$ is a $p^n$-th power.  
 
 A negative answer to this question will follow from the following  important result of D. Bubboloni and G. Corsi Tani (see \cite{Bub}).  

Recall that a $p$-central group is a group in which every element of order $p$ is central.  D. Bubboloni and G. Corsi Tani used the term TH-group instead of $p$-central group, where   TH-group refers to J. G. Thompson, who seems to be the first to observe the importance of $p$-central groups (see \cite[Hilfssatz III.12.2]{Hup}).    
\begin{thm}\label{Bub}
Let be $d$ and $n$ two positive integers, $p$ an odd prime, and $F$  the free groupe on $d$ generators.  Then $G_n=F/\lambda_{n+1}^p(F)$ is a $p$-central $p$-group.  More precisely we have $\Omega_1(G_n)=\lambda_{n}^p(F)/ \lambda_{n+1}^p(F)$.  
\end{thm}
\begin{cor}
A finite $p$-group, $p$ odd, which satisfies $\Omega_1(\gamma_{p-1}(G))\leq \operatorname{Z}(G)$ needs not be necessarily power closed.   
\end{cor}
\begin{proof}
Obviously a $p$-central $p$-group satisfies $\Omega_1(\gamma_{p-1}(G))\leq \operatorname{Z}(G)$.  Assume for a contradiction that the result is false.  So the $p$-groups $G_n$ are all power closed. Now every finite $p$-group is a quotient of some $G_n$ (for appropriates $n$ and $d$),  and the property of being power closed is inherited by quotients. It follows  that every finite $p$-group is power closed, which is a contradiction.
\end{proof}

Let us mention  briefly  another consequence of  Theorem \ref{Bub}.  It is largely believed that finite $p$-central $p$-groups are dual (in a sense) to powerful $p$-groups.  Since the inverse limits of powerful $p$-groups have  roughly a uniform structure (see \cite[\S 3]{DSMS}), which qualifies them to play a central role in studying analytic pro-$p$ groups (see \cite{DSMS}), it is natural to ask if there is a restriction on the structure of a {\itshape pro- $p$-central $p$-group}, that is an inverse limit of finite  $p$-central $p$-groups.   

Let $F$ be the free group on a finite number of generators.  As every normal subgroup of $F$ of $p$-power index contains a subgroup $\lambda_{n}^p(F)$ for some $n$, it follows that
$$\widehat{F}_p \cong \lim_{\longleftarrow}  F/\lambda_{n}^p(F)$$ 
where $\widehat{F}_p$ is the pro-$p$ completion of $F$.  This shows that the free pro-$p$ group $\widehat{F}_p$ is pro $p$-central, so there is no reasonable restriction on the structure of a (finitely generated)  pro $p$-central $p$-group.\\

\begin{center}{\textbf{Acknowledgments}}
\end{center}

This work could not be completed without the help of Daniela Bubbolon, Urban Jezernik and Andrea Caranti.  We are really grateful to them.

\address{ Department of Mathematics, University Kasdi Merbah Ouargla Ouargla, Algeria. \\ {\tt Email: benmoussageo@gmail.com}  \\
{\tt Email: yassine\_guer@hotmail.fr} }

\end{document}